\documentclass[a4paper]{scrartcl}
\usepackage[latin1]{inputenc}
\usepackage{amsmath}
\usepackage{bbm}
\usepackage{amssymb}
\usepackage{amsthm}
\usepackage{pdfsync}
\usepackage{esint}
\usepackage{graphicx}
\usepackage{epstopdf}
\newtheorem{lemma}{Lemma}[section]
\newtheorem{theorem}{Theorem}[section]

\theoremstyle{remark}
\newtheorem{remark}{Remark}[section]
\DeclareMathOperator{\grad}{grad}
\DeclareMathOperator{\Hess}{Hess}

\numberwithin{equation}{section}

\begin{document}
\title{Asymptotics for characteristic polynomials of Wishart type products of complex Gaussian and truncated unitary random matrices}
\author{Thorsten Neuschel, Dries Stivigny
  \thanks{Department of Mathematics, KU Leuven, Celestijnenlaan 200B box 2400, BE-3001 Leuven, Belgium. This work is supported by KU Leuven research grant OT\slash12\slash073 and the Belgian Interuniversity Attraction Pole P07/18. E-mail: Thorsten.Neuschel@wis.kuleuven.be, Dries.Stivigny@wis.kuleuven.be}}

 \date{\today}

\maketitle

\paragraph{Abstract} Based on the multivariate saddle point method we study the asymptotic behavior of the characteristic polynomials associated to Wishart type random matrices that are formed as products consisting of independent standard complex Gaussian and a truncated Haar distributed unitary random matrix. These polynomials form a general class of hypergeometric functions of type \(_2 F_r\). We describe the oscillatory behavior on the asymptotic interval of zeros by means of formulae of Plancherel-Rotach type and subsequently use it to obtain the limiting distribution of the suitably rescaled zeros. Moreover, we show that the asymptotic zero distribution lies in the class of Raney distributions and by introducing appropriate coordinates elementary and explicit characterizations are derived for the densities as well as for the
distribution functions.


\section{Introduction}

The behavior of the eigenvalues of random matrices is a large subject of research in random matrix theory. Recently, the study of products of random matrices gained particular interest (see, e.g., \cite{Adhi}, \cite{Akemann1}, \cite{Burda1}, \cite{Forrester0}, \cite{Forrester1}, \cite{KuijlaarsZhang}, \cite{KuijlaarsStivigny}, \cite{Neuschel2}). Let \(r\in\mathbb{N}=\{1, 2, 3, \ldots\}\) be an arbitrary positive integer and denote by \(G_1, G_2, \ldots, G_r\) independent standard complex Gaussian random matrices (matrices of this kind are called Ginibre random matrices). Moreover, let each matrix \(G_j\) be of dimension \(N_j \times N_{j-1}\) and let the matrix \(Y_r\) be defined as the product
\begin{equation}\label{1}Y_r=G_r G_{r-1} \cdots G_1.
\end{equation}
Assuming \(N_0 = \min\{N_0, \ldots, N_r\}\) and writing \(n=N_0\), let us consider the \(n \times n\)-dimensional matrix \(Y_r^{\ast} Y_r\), where \(Y_r^{\ast}\) denotes the conjugate transpose of \(Y_r\). It was shown by Akemann, Ipsen and Kieburg in \cite{Akemann2} that the eigenvalues of \(Y_r^{\ast} Y_r\) form a determinantal point process with a correlation kernel expressible in terms of Meijer G-functions. Moreover, Kuijlaars and Zhang showed in \cite{KuijlaarsZhang} that this point process can be interpreted as a multiple orthogonal polynomial ensemble. The average characteristic polynomials of the matrices \(Y_r^{\ast} Y_r\) are given as generalized hypergeometric polynomials of the form
\[ (-1)^n \prod_{l=1}^{r} (\nu_l +1)_n  ~ _{1} F_{r} \left(\begin{matrix}
 & -n& \\ \nu_1 +1, & \ldots, &\nu_r +1 & \end{matrix}\,\bigg\vert\,  x \right),
\]
where \(\nu_j = N_j -N_0\) for \(j\in\{1,\ldots,r\}\) (see \cite{KuijlaarsZhang}, \cite{Akemann2}). These polynomials have been studied in \cite{Neuschel2} with respect to their behavior on the region of zeros (after suitable rescaling) in form of an asymptotic formula of Plancherel-Rotach type. Moreover, using this representation it was shown that the asymptotic zero distribution is given by the Fuss-Catalan distribution of order \(r\) (which matches with the known fact that the macroscopic density of the eigenvalues of the matrices \(Y_r^{\ast} Y_r\) is given by the Fuss-Catalan distribution).

As well as the consideration of products solely consisting of Ginibre matrices it is of interest to study products involving factors with different distributions. This has been done, for instance, in \cite{Forrester0} by Forrester, where products of complex Gaussian and inverse complex Gaussian matrices are studied. In this context it is also interesting to involve Haar distributed unitary factors (see, e.g., \cite{Forrester3}). Recently, in \cite{KuijlaarsStivigny} the authors considered the squared singular values of products of the type in \eqref{1} in which the first matrix \(G_1\) is replaced by a truncated unitary random matrix \(X\). More precisely, let \(r>1\), \(U\) be a Haar distributed unitary \(l\times l\) matrix and let \(X\) be the \((n+\nu_1)\times n\) upper left block of \(U\), where \(\nu_1 \geq 0\) and \(l\geq 2n+\nu_1\). Now let us consider the product of independent matrices
\begin{equation}\label{2}Z_r  =G_r G_{r-1} \cdots G_2 X,
\end{equation}
where \(G_j\) are Ginibre matrices of size \((n+\nu_j)\times(n+\nu_{j-1})\), \(\nu_j\geq 0\). It is shown in \cite{KuijlaarsStivigny} that the squared singular values of \(Z_r\) form a determinantal point process with joint probability distribution on \((0,\infty)^n\) given by a density (with respect to the Lebesgue measure) proportional to
\[\prod_{1\leq j<k\leq n}(x_k-x_j) \det\left[w_{k-1}(x_j)\right]_{j,k=1}^n,\]
where the functions \(w_k\) are given as Meijer G-functions by
\[w_{k}(x)=G_{1,r}^{r,0} \left(\begin{matrix}
   l-2n+1+k \\ \nu_r ,  \ldots, \nu_{2},  \nu_1 +k \end{matrix}\,\bigg\vert\,  x \right).\]
Moreover, the average characteristic polynomials of the Wishart type random matrices \(Z_r^{*}Z_r\) are given by the generalized hypergeometric polynomials \cite{KuijlaarsStivigny}
\[(-1)^n \prod_{i=1}^{r} \frac{\Gamma(n+1+\nu_i)}{\Gamma(\nu_i +1)} \frac{\Gamma(l-n+1)}{\Gamma(l+1)}  ~ _{2} F_{r} \left(\begin{matrix}
  -n,&l-n+1 \\ \nu_1 +1, & \ldots, &\nu_r +1 & \end{matrix}\,\bigg\vert\,  x \right). \]
In this paper we study the behavior of these polynomials and their zeros for large values of \(n\), where we consider \(\nu_j\geq 0\) and \(\kappa =l-2n+1\geq 0\) as fixed integers. This means that the dimensions of the associated matrices grow to infinity in a way described by the fixed differences \(\nu_j\) of the sizes. Moreover, for the sake of a more convenient analysis, we consider one of the parameters \(\nu_j\) to be zero, where we (arbitrarily) choose \(\nu_r=0\). Thus, here the average characteristic polynomials are given by
\begin{align*}P_n(x)=&(-1)^n \prod_{i=1}^{r} \frac{\Gamma(n+1+\nu_i)}{\Gamma(\nu_i +1)} \frac{\Gamma(\kappa+n)}{\Gamma(\kappa +2n)}  ~ _{2} F_{r} \left(\begin{matrix}
  &-n,&n+\kappa \\ \nu_1 +1, & \ldots, &\nu_{r-1} +1, & 1 \end{matrix}\,\bigg\vert\,  x \right)\\
  =& (-1)^n n! \prod_{i=1}^{r-1} \Gamma(n+1+\nu_i) \frac{\Gamma(\kappa+n)}{\Gamma(\kappa +2n)}  F_n(x),
\end{align*}
where we introduce the polynomials \(F_n\) by
\begin{equation}\label{3}F_n(x)=\sum_{k=0}^n \binom{n}{k}\frac{(n+\kappa)_k (-x)^k}{k! (\nu_1+k)! \cdots (\nu_{r-1}+k)!}.
\end{equation}
In many cases there is a close connection between random matrices and their average characteristic polynomials. For instance, one expects that (after proper rescaling) the limiting distribution of the zeros of the characteristic polynomials coincides with the macroscopic density of eigenvalues (see, e.g., \cite{Hardy}). It will emerge from our analysis that this expectation holds true for the class of Wishart type random matrices \(Z_r^{*}Z_r\) (see \eqref{2}) under consideration here.

The paper is structured as follows: After stating some auxiliary results in Section 2 we derive the large \(n\) behavior of the suitably rescaled polynomials \(F_n\) on the region of zeros in form of an asymptotic formula of Plancherel-Rotach type. More precisely, for \(r>1\), \(\nu_1,\ldots,\nu_{r-1} \in \mathbb{N}_0\) arbitrary non-negative integers and \(\kappa =l-2n+1\geq 0\) as described above, we show in Theorem \ref{PRA} the following: If we parameterize the asymptotic interval of zeros by \(x=\sigma(\varphi)\) using
\[\sigma:\left(0,\frac{\pi}{r+1}\right)\rightarrow\left(0,\frac{(r+1)^{(r+1)/2}}{2(r-1)^{(r-1)/2}}\right),\]
\[\sigma(\varphi)=\frac{\left(\sin{(r+1)\varphi}\right)^{(r+1)/2}}{\sin{2\varphi}\left(\sin{(r-1)\varphi}\right)^{(r-1)/2}},\] 
then we obtain the asymptotic formula of Plancherel-Rotach type

\begin{align}\nonumber  F_n (n^{r-1} x)=&\frac{2(-1)^n}{(2\pi)^{r/2}}\left(a(\varphi)n\right)^{-r/2-(\nu_1+\ldots+\nu_{r-1})}\left(1+2a(\varphi)\cos\varphi +a(\varphi)^2\right)^{\kappa/2} \\\nonumber
\times&\left((r+1)^2-2(r^2-1)a(\varphi)^2\cos{2\varphi}+(r-1)^2 a(\varphi)^4\right)^{-1/4}\\ \label{4}
\times&\exp\left\{n a(\varphi)(r-1)\cos\varphi\right\}\left(\frac{\sqrt{(1-a(\varphi)^2)^2+(2a(\varphi)\sin\varphi)^2}}{1+a(\varphi)^2-2a(\varphi)\cos\varphi}\right)^n\\\nonumber
\times&\left\{\cos{\left(n\,f(\varphi)+g(\varphi)\right)}+o(1)\right\},
\end{align}
as \(n\rightarrow \infty\), where we introduce the expressions for \(0<\varphi<\frac{\pi}{r+1}\)
\begin{equation}\label{a}a(\varphi)=\left(\frac{\sin{(r+1)\varphi}}{\sin{(r-1)\varphi}}\right)^{1/2},\end{equation}
\begin{align*}&f(\varphi)=\frac{\pi}{2}-(r-1)a(\varphi)\sin\varphi+\arctan\left(\frac{1-a(\varphi)^2}{2a(\varphi)\sin\varphi}\right),\\
&g(\varphi)=\left(\frac{\pi}{2}+\nu_1+\cdots+\nu_{r-1}\right)\varphi-\kappa \arctan\left(\frac{a(\varphi)\sin(\varphi)}{1+ a(\varphi)\cos(\varphi)}\right)\\
&\quad\quad\quad-\frac{1}{2} \arctan\left(\frac{(r-1) a(\varphi)^2 \sin(2\varphi)}{r+1-(r-1)a(\varphi)^2 \cos(2\varphi)}\right).
\end{align*}

As one important application, in Section 3 we will use \(\eqref{4}\) to find explicit representations for the limiting distribution of the zeros, which turns out to coincide with the so-called Raney distribution \(R_{\frac{r+1}{2}, \frac{1}{2}}\) (see Theorem \ref{WA}). This distribution is supported on the interval \(\left[0, \frac{(r+1)^{(r+1)/2}}{2(r-1)^{(r-1)/2}}\right]\) and given by the moments
\[\frac{1}{(r+1)n+1} \binom{\frac{1}{2}((r+1)n+1)}{n}.\]
Moreover, we are able to obtain an elementary and explicit description for the density \(v\) of \(R_{\frac{r+1}{2}, \frac{1}{2}}\) and its distribution function in suitable coordinates (see Theorem \ref{C}). More precisely, if
 \[x=\sigma(\varphi)=\frac{\left(\sin{(r+1)\varphi}\right)^{(r+1)/2}}{\sin{2\varphi}\left(\sin{(r-1)\varphi}\right)^{(r-1)/2}},\quad 0<\varphi<\frac{\pi}{r+1},\]
then we have
\[v(x)=\frac{\sin 2\varphi \sin \varphi (\sin (r-1)\varphi)^{\frac{r}{2}-1}}{\pi (\sin(r+1)\varphi)^{\frac{r}{2}}}.\]

It is important to remark that recently in \cite{Forrester3} the applied techniques have been used by Forrester and Liu to obtain similar kinds of densities for the general class of Raney distributions.

As the present work complements the results in \cite{Neuschel2} by extending them to hypergeometric polynomials of the type \(_{2} F_{r}\), the methods we use here are an adaptation of the methods applied in \cite{Neuschel2}, which requires a considerable amount of new difficulties to overcome. Moreover, the stated results may be useful to study questions of universality (in the sense of random matrix theory) concerning the correlation kernels involved.

\section{Plancherel-Rotach formula for the average characteristic polynomials}

At the beginning of this section we state some auxiliary results. The first one we mention is a simple version of the multivariate method of saddle points (see \cite{Neuschel} for a short proof and discussion).
\begin{theorem}\label{MSP}
Let \(p\) and \(q\) be holomorphic functions on a complex domain \(D\subset \mathbb{C}^r\) with \([-a,a]^r \subset D\) for a number \(a>0\), and let
\[I(n)=\int\limits_{[-a,a]^r} e^{-n p(t)} q(t) dt,\]
where \(t=(t_1,\ldots,t_r)\). Moreover, let \(t=0\) be a simple saddle point of the function \(p\), which means that we have for the complex gradient
\[\grad p (0) = 0\]
and for the Hessian
\[\det\Hess p (0) \neq 0.\]
Furthermore, suppose that, considered as a real-valued function on \([-a,a]^r\), \(\Re [p (w)]\) attains its minimum exactly at the point \(t=0\) with \(\det\Re\Hess p (0) \neq 0\) and \(q(0)\neq 0\). Then we have
\begin{equation}\label{A}I(n) = \left(\frac{2\pi}{n}\right)^{r/2} e^{-n p(0)} \frac{q(0)}{\sqrt{\det\Hess p (0)}} (1+o(1)),
\end{equation}
as \(n\rightarrow \infty\).
\end{theorem}

\begin{remark} The branch of the square root in (\ref{A}) is determined by the identity
\[\int\limits_{\mathbb{R}^r} e^{-\frac{1}{2} t^T\Hess p (0)t}dt=\frac{(2\pi)^{r/2}}{\sqrt{\det\Hess p (0)}}.\]
In general, the proper choice of the branch for the square root in (\ref{A}) can be described by fixing the arguments of the eigenvalues of \(\Hess p (0)\) in a correct manner (see, e.g., \cite{Fedoryuk}), which is related to the Maslov index.
\end{remark}

Next we prove a preliminary result on the location of the zeros of the polynomials in question.
\begin{lemma}\label{Zeros} Let \(r\in\mathbb{N}\) and \(\nu_1, \ldots, \nu_r ,\kappa \in \mathbb{N}_0\) be arbitrary integers, then all zeros of the polynomials \(F_n\) defined in (\ref{3}) are real and positive.
\end{lemma}
\begin{proof} The positivity of the zeros clearly results from their interpretation as averages of positive eigenvalues. Nevertheless, we provide a short proof relying on analytical facts only. We start off by observing that all zeros of the following polynomials are real and positive
\[\sum_{k=0}^n \binom{n}{k}\frac{(n+\kappa)_k (-x)^k}{k!}=P_n^{(0,\kappa-1)}\left(1-2x\right),\]
as on the right-hand side we have Jacobi polynomials with parameters \(0\) and \(\kappa-1\geq-1\) evaluated at \(1-2x\). From this we can conclude for an arbitrary interger \(m\geq 0\) that all zeros of
\[\sum_{k=0}^n \binom{n}{k}\frac{(n+\kappa)_k (-x)^k}{k! (k+m)!}\]
are real and positive (use \cite{Polya2}, Part 5, Chap. 1, Probl. 63). Hence, by mathematical induction we can increase the number of factorials in the denominator of the coefficients and we immediately obtain the statement.
\end{proof}

Finally, we provide an inequality for a function of several variables which will be crucial for the proof of the main result in Theorem \ref{PRA}.
\begin{lemma}\label{Inequality} For \(0<\varphi<\frac{\pi}{r+1}\) let the function \(a(\varphi)\) be defined by (\ref{a}) and let the function \(b(\varphi)\) be defined by
\[b(\varphi)=\frac{a(\varphi)}{\left(1+2a(\varphi)\cos\varphi +a(\varphi)^2\right)^{1/2}}.\]
Moreover, let the function \(h:[-\pi,\pi]^r\rightarrow \mathbb{R}\) be defined by
\[h(t_1,\ldots,t_r)=\left\vert\exp\left\{a(\varphi)\sum_{j=2}^r e^{i t_j}\right\}\left(1-b(\varphi)e^{it_1}\right)^{-1}\left(1-\frac{\sin(r+1)\varphi}{b(\varphi)\sin2\varphi}\exp\left\{-i\sum_{j=1}^r t_j\right\}\right)\right\vert.\]
Then the function \(h\) attains its global maximum exactly in the two points 
\[\left(\arctan\left(\frac{\sin\varphi}{\cos\varphi + a(\varphi)}\right),\varphi,\ldots,\varphi\right)\]
and
\[-\left(\arctan\left(\frac{\sin\varphi}{\cos\varphi + a(\varphi)}\right),\varphi,\ldots,\varphi\right).\]
\end{lemma}
\begin{proof} Let \(r>1\) and \(0<\varphi<\frac{\pi}{r+1}\) be fixed. For \(t=(t_1,\ldots,t_r)\in [-\pi,\pi]^r\) we have 
\begin{align*}h(t)=&\exp\left\{a(\varphi)\sum_{j=2}^r\cos t_j\right\} \left\vert\frac{1-\frac{\sin(r+1)\varphi}{b(\varphi)\sin 2\varphi} \exp\left\{-i\sum_{j=1}^r t_j \right\}}{1-b(\varphi)e^{i t_1}}\right\vert\\
=&\frac{1}{b(\varphi)} \exp\left\{a(\varphi) \sum_{j=2}^r \cos t_j \right\}\left\vert\frac{e^{i t_1}-\frac{\sin(r+1)\varphi}{b(\varphi)\sin 2\varphi} \exp\left\{-i\sum_{j=2}^r t_j \right\}}{e^{i t_1}-\frac{1}{b(\varphi)}}\right\vert.
\end{align*}
Now, for \(c \in [-(r-1)\pi,(r-1)\pi]\) we define the hyperplane 
\[H_c =\left\{(t_2,\ldots,t_r)\in [-\pi,\pi]^{r-1} \,:\, \sum_{j=2}^r t_j = c\right\}.\]
As can be seen by means of elementary calculus the restriction of the function
\[\sum_{j=2}^r \cos t_j\]
to the hyperplane \(H_c\) attains its global maximum exactly at the point \((t_2,\ldots,t_r)=\left(\frac{c}{r-1},\ldots,\frac{c}{r-1}\right)\)
with value \((r-1)\cos\frac{c}{r-1}\). Thus, we obtain
\begin{align*}&\max_{t\in [-\pi,\pi]^r} h(t)\\
&=\max_{t_1 \in [-\pi,\pi]}~~ \max_{c \in [-(r-1)\pi,(r-1)\pi]} \frac{1}{b(\varphi)}\exp\left\{a(\varphi) (r-1)\cos\frac{c}{r-1} \right\}\left\vert\frac{e^{i t_1}-\frac{\sin(r+1)\varphi}{b(\varphi)\sin 2\varphi} e^{-ic}}{e^{i t_1}-\frac{1}{b(\varphi)}}\right\vert.
\end{align*}
Following the arguments in the proof of Lemma 2.2 in \cite{Neuschel2}, we can restrict the interval of \(c\) to \([-\pi,\pi]\), so that we have to study the function
\[\tilde{h}(t_1,c)=\exp\left\{a(\varphi) (r-1)\cos\frac{c}{r-1} \right\}\left\vert\frac{e^{i t_1}-\frac{\sin(r+1)\varphi}{b(\varphi)\sin 2\varphi} e^{-ic}}{e^{i t_1}-\frac{1}{b(\varphi)}}\right\vert\]
on the domain \([-\pi,\pi]^2\). We will show that \(\tilde{h}\) attains its global maximum exactly in the two points 
\[\left(\arctan\left(\frac{\sin\varphi}{\cos \varphi +a(\varphi)}\right), (r-1)\varphi\right)\]
and 
\[-\left(\arctan\left(\frac{\sin\varphi}{\cos \varphi +a(\varphi)}\right), (r-1)\varphi\right).\]
To this end, by symmetry we can restrict the considerations to the domain \([0,\pi]^2\) and introduce the M\"{o}bius transform
\[L(z)=\frac{z-\alpha}{z-\beta},\]
where we have 
\[\alpha=\frac{\sin(r+1)\varphi}{b(\varphi)\sin 2\varphi} e^{-ic} \]
and
\[\beta=\frac{1}{b(\varphi)}>1.\]
It is not difficult to observe that \(L\) maps the unit circle to a circle with center 
\[\frac{\alpha\beta-1}{\beta^2-1}\]
and radius 
\[\frac{\vert\alpha-\beta\vert}{\beta^2-1}.\]
Thus, for a fixed \(c\in[0,\pi]\) the maximum of \(\tilde{h}(t_1,c)\) will be attained at a unique point \(t_1\in[0,\pi]\) with 
\begin{align*}\max_{t_1\in [0,\pi]} \tilde{h}(t_1,c)=&\exp\left\{a(\varphi) (r-1)\cos\frac{c}{r-1} \right\}\frac{1}{\beta^2-1}\\
&\times \left\{\left\vert\frac{1}{b(\varphi)}-\frac{\sin(r+1)\varphi}{b(\varphi)\sin 2\varphi} e^{-ic}\right\vert+\left\vert1-\frac{\sin(r+1)\varphi}{b(\varphi)^2\sin 2\varphi} e^{-ic}\right\vert\right\}.
\end{align*}
Let us denote the right-hand side of the latter expression by \(h^{\ast}(c)\) and study its extremal points. For the derivative we obtain
\begin{align*}\frac{d}{dc}h^{\ast}(c)=&\frac{\sin\frac{c}{r-1}}{\beta^2-1}\exp\left\{a(\varphi) (r-1)\cos\frac{c}{r-1} \right\}\\
&\times\left\{\left\vert\frac{1}{b(\varphi)}-\frac{\sin(r+1)\varphi}{b(\varphi)\sin 2\varphi} e^{-ic}\right\vert^{-1}+\left\vert1-\frac{\sin(r+1)\varphi}{b(\varphi)^2\sin 2\varphi} e^{-ic}\right\vert^{-1}\right\}\\
&\times\left\{\frac{\sin(r+1)\varphi \sin c}{b(\varphi)^2 \sin 2\varphi \sin\frac{c}{r-1}}-a(\varphi)\left\vert\frac{1}{b(\varphi)}-\frac{\sin(r+1)\varphi}{b(\varphi)\sin 2\varphi} e^{-ic}\right\vert \left\vert1-\frac{\sin(r+1)\varphi}{b(\varphi)^2\sin 2\varphi} e^{-ic}\right\vert\right\}.
\end{align*}
Now we can observe that the last bracket in the latter expression is a strictly decreasing function of \(c\) on \((0,\pi)\) which starts with a positive value and ends with a negative value. Moreover, it can be checked that it vanishes at the point \(c=(r-1)\varphi\), which shows that \(h^{\ast}\) attains its maximum exactly at this point. Finally, by a further computation it follows that the partial derivative \(\frac{\partial}{\partial t_1} \tilde{h} (t_1, (r-1)\varphi)\) vanishes at \(t_1=\arctan\left(\frac{\sin\varphi}{\cos\varphi+a(\varphi)}\right)\), from which the statement follows.

\end{proof}

Now we turn to the first main result which describes the asymptotic behavior of the rescaled polynomials \(F_n (n^{r-1} x)\).

\begin{theorem}[Asymptotics of Plancherel-Rotach type]\label{PRA} Let \(r\in\mathbb{N}\), \(r>1\) and \(\nu_1, \ldots, \nu_{r-1}, \kappa \in \mathbb{N}_0\) be arbitrary integers, then we have 
\begin{align}\nonumber  F_n (n^{r-1} x)=&\frac{2(-1)^n}{(2\pi)^{r/2}}\left(a(\varphi)n\right)^{-r/2-(\nu_1+\ldots+\nu_{r-1})}\left(1+2a(\varphi)\cos\varphi +a(\varphi)^2\right)^{\kappa/2} \\\nonumber
\times&\left((r+1)^2-2(r^2-1)a(\varphi)^2\cos{2\varphi}+(r-1)^2 a(\varphi)^4\right)^{-1/4}\\ \label{5}
\times&\exp\left\{n a(\varphi)(r-1)\cos\varphi\right\}\left(\frac{\sqrt{(1-a(\varphi)^2)^2+(2a(\varphi)\sin\varphi)^2}}{1+a(\varphi)^2-2a(\varphi)\cos\varphi}\right)^n\\\nonumber
\times&\left\{\cos{\left(n\,f(\varphi)+g(\varphi)\right)}+o(1)\right\},
\end{align}
as \(n\rightarrow \infty\), where we have for \(0<\varphi<\frac{\pi}{r+1}\)
\begin{equation}x=\sigma(\varphi)=\frac{\left(\sin{(r+1)\varphi}\right)^{(r+1)/2}}{\sin{2\varphi}\left(\sin{(r-1)\varphi}\right)^{(r-1)/2}},
\end{equation}
\begin{equation}\label{a2}a(\varphi)=\left(\frac{\sin{(r+1)\varphi}}{\sin{(r-1)\varphi}}\right)^{1/2},\end{equation}
\begin{align*}&f(\varphi)=\frac{\pi}{2}-(r-1)a(\varphi)\sin\varphi+\arctan\left(\frac{1-a(\varphi)^2}{2a(\varphi)\sin\varphi}\right),\\
&g(\varphi)=\left(\frac{\pi}{2}+\nu_1+\cdots+\nu_{r-1}\right)\varphi-\kappa \arctan\left(\frac{a(\varphi)\sin(\varphi)}{1+ a(\varphi)\cos(\varphi)}\right)\\
&\quad\quad\quad-\frac{1}{2} \arctan\left(\frac{(r-1) a(\varphi)^2 \sin(2\varphi)}{r+1-(r-1)a(\varphi)^2 \cos(2\varphi)}\right).
\end{align*}
\end{theorem}
\begin{proof}At first we establish a representation for the polynomials \(F_n\) as a multivariate complex contour integral which is suitable for determining the asymptotic behavior using the multivariate method of saddle points. To this
end, we express the coefficients of the polynomial \(F_n\) considering single contour integrals of the form
\[\frac{1}{N!}=\frac{1}{2\pi i} \int \frac{e^{z}}{z^{N+1}} dz,\]
\[\frac{(n+\kappa)_k}{k!}=\frac{1}{2\pi i} \int \frac{(1-z)^{-n-\kappa}}{z^{k+1}} dz,\]
where the paths of integration are positive-oriented curves around the origin (with radius less than one in the second integral). Using the binomial theorem we obtain
\begin{align*}&F_n(x)\\
=&\frac{1}{(2\pi i)^r} \int\limits_{\Gamma} \exp\left\{w_2+\ldots+w_r\right\} (1-w_1)^{-n} \left(1-\frac{x}{w_1\ldots w_r}\right)^n\frac{(1-w_1)^{-\kappa}}{w_1 w_2^{\nu_1+1}\ldots w_r^{\nu_{r-1}+1}} dW,
\end{align*}
with \(W=(w_1,\ldots,w_r)\), \(\Gamma=\gamma_1 \times \ldots \times \gamma_r\), and \(\gamma_j\) are positive-oriented contours around the origin (and \(\gamma_1\) has radius less than one). By a change of variables \(w_j \mapsto n w_j\) for \(j=2,\ldots, r\) and replacing \(x\) by \(n^{r-1}x\) we obtain
\begin{equation}\label{int} F_n(n^{r-1}x)=\frac{n^{-(\nu_1+\ldots+\nu_{r-1})}}{(2\pi i)^r}\int\limits_{\Gamma}\left\{H(w_1,\ldots,w_{r})\right\}^nQ(w_1,\ldots,w_{r})dW,
\end{equation}
where
\[H(w_1,\ldots,w_{r})=e^{w_2+\ldots+w_{r}}\frac{1}{1-w_1}\left(1-\frac{x}{w_1\cdot\ldots\cdot w_{r}}\right),\]
\[Q(w_1,\ldots,w_{r})=\frac{(1-w_1)^{-\kappa}}{w_1 w_2^{\nu_{1}+1}\cdot\ldots\cdot w_{r}^{\nu_{r-1}+1}}.\]
In order to study the multivariate saddle points of \(H\), we compute the complex partial derivatives
\[H_{w_1}=\frac{e^{w_2+\ldots +w_{r}}}{(1-w_1)^2 w_1^2 w_2\cdots w_r}\left\{w_1^2 w_2\cdots w_r-2 x w_1 +x\right\},\]
and for \(j=2,\ldots,r\)
\[H_{w_j}=\frac{e^{w_2+\ldots +w_{r}}}{(1-w_1) w_1 \cdots w_j^2 \dots w_r}\left\{w_1 \cdots w_j^2 \dots w_r-  x w_j +x\right\}.\]
As the multivariate saddle points are solutions of the equation \(\grad H = 0\), it is not difficult to see that every saddle point \((w_1,\ldots,w_r)\) is of the general form
\[(w_1,\ldots,w_r)=\left(\frac{w}{w+1},w,\ldots,w\right),\]
where \(w\) satisfies the algebraic equation
\begin{equation}\label{eq} w^{r+1}-w^2 x+x=0.
\end{equation}
Now, introducing polar coordinates for \(w\) and carefully studying the imaginary and the real part of equation (\ref{eq}) it turns out that using the parametrization
\[x=\sigma(\varphi)=\frac{\left(\sin{(r+1)\varphi}\right)^{(r+1)/2}}{\sin{2\varphi}\left(\sin{(r-1)\varphi}\right)^{(r-1)/2}},~~~0<\varphi<\frac{\pi}{r+1},\]
two roots of (\ref{eq}) are located at the points
\[w=a(\varphi) e^{i \varphi}\quad\quad\text{and}\quad\quad w=a(\varphi) e^{-i \varphi},\]
where \(a(\varphi)\) is defined in (\ref{a2}). Writing \(w(\varphi)=a(\varphi) e^{i \varphi}\), this means for \(x=\sigma(\varphi)\) we obtain two complex conjugate multivariate saddle points at
\begin{equation}\label{Saddles}\left(\frac{w(\varphi)}{w(\varphi)+1},w(\varphi),\ldots,w(\varphi)\right)\quad \text{and}\quad \left(\overline{\frac{w(\varphi)}{w(\varphi)+1}},\overline{w(\varphi)},\ldots,\overline{w(\varphi)}\right).
\end{equation}
A small computation shows 
\[\frac{w(\varphi)}{w(\varphi)+1}=b(\varphi) e^{i\theta},\]
where \(b(\varphi)\) is defined as in the statement of Lemma \ref{Inequality} by
\[b(\varphi)=\frac{a(\varphi)}{\left(1+2a(\varphi)\cos\varphi +a(\varphi)^2\right)^{1/2}}\]
and we have
\[\theta=\arctan\left(\frac{\sin(\varphi)}{\cos(\varphi)+a(\varphi)}\right).\]
Using the parameterizations 
\[w_1=b(\varphi)e^{it_1},~ t_1\in[-\pi,\pi]\]
and for \(j=2,\ldots,r\)
\[w_j=a(\varphi)e^{it_j},~ t_j\in[-\pi,\pi],\]
from (\ref{int}) we obtain for \(x=\sigma(\varphi)\) the integral representation
\begin{equation}\label{int2} F_n(n^{r-1}x)=\frac{n^{-(\nu_1+\ldots+\nu_{r-1})}}{(2\pi i)^r}\int\limits_{[-\pi,\pi]^r}\left\{\tilde{H}(t_1,\ldots,t_{r})\right\}^n \tilde{Q}(t_1,\ldots,t_{r})dT,
\end{equation}
where \(T=(t_1,\ldots,t_r)\),
\[\tilde{H}(t_1,\ldots,t_{r})=\exp\left\{a(\varphi)\sum_{j=2}^r e^{i t_j}\right\}\left(1-b(\varphi)e^{it_1}\right)^{-1}\left(1-\frac{\sin(r+1)\varphi}{b(\varphi)\sin2\varphi}\exp\left\{-i\sum_{j=1}^r t_j\right\}\right),\]
and 
\[\tilde{Q}(t_1,\ldots,t_r)=Q\left(b(\varphi) e^{it_1},a(\varphi)e^{i t_2}\ldots,a(\varphi)e^{i t_r}\right).\]
From Lemma \ref{Inequality} we know that the modulus \(h(t_1,\ldots,t_r)=\vert\tilde{H}(t_1,\ldots,t_r)\vert\) on \([-\pi,\pi]^r\) attains its global maximum value exactly at the two points corresponding to the saddle points in (\ref{Saddles}). Moreover, taking the geometry of the integrand in (\ref{int2}) into account, the contributions coming from both of these saddle points will be complex conjugates. So in order to establish the asymptotic behavior of (\ref{int2}), we can restrict our attention to a small neighbourhood \(U\) of the point 
\[S(\varphi)=\left(\arctan\left(\frac{\sin(\varphi)}{\cos(\varphi)+a(\varphi)}\right),\varphi,\ldots,\varphi\right).\]
On the neighbourhood \(U\) we can study the integral 
\[\int\limits_{U} e^{-n p(t_1,\ldots,t_r)} \tilde{Q}(t_1,\ldots,t_{r})dT,\]
where we have according to the definition of \(\tilde{H}\)
\[p(t_1,\ldots,t_r)=-a(\varphi)\sum_{j=2}^r e^{i t_j}+\log\left(1-b(\varphi)e^{i t_1}\right)-\log\left(1-\frac{\sin(r+1)\varphi}{b(\varphi)\sin (2\varphi)} \exp\left\{-i \sum_{j=1}^r t_j\right\}\right).\]
We have by construction \(\grad p (S(\varphi))=0\) and an elementary calculation yields for the second partial derivatives
\[\frac{\partial^2 p}{\partial t_1^2}(S(\varphi))=2w(\varphi),\]
\[\frac{\partial^2 p}{\partial t_j^2}(S(\varphi))=w(\varphi)-w(\varphi)(w(\varphi)-1),\quad j=2,\ldots,r,\]
and
\[\frac{\partial^2 p}{\partial t_j \partial t_k}(S(\varphi))=-w(\varphi)(w(\varphi)-1),\quad j\neq k.\]
Now using the determinantal identity
\[\det\begin{pmatrix}
a &  c  & \ldots & c\\
c  &  b & \ldots & c\\
\vdots & \vdots & \ddots & \vdots\\
c  &   c       &\ldots & b
\end{pmatrix}=(-1)^{r-1}\left(c-b\right)^{r-2}\left((r-1)c^2 -(r-2)ac-ab\right),
\]
where the determinant of an \(r \times r\) matrix is taken, we can explicitly evaluate the determinant of the Hessian of \(p\) at \(S(\varphi)\) as
\[\det \Hess p (S(\varphi))=w(\varphi)^r\left(r+1-(r-1)w(\varphi)^2\right)\neq 0,\]
for \(0<\varphi<\frac{\pi}{r+1}\). Moreover, in the same manner we obtain for the determinants of the real parts of the Hessian
\[\det\Re\Hess p(S(\varphi))=a(\varphi)^r\left(\cos \varphi\right)^{r-2}\left\{(r+1)\cos(\varphi)^2-(r-1)\frac{\sin(r+1)\varphi}{\sin(r-1)\varphi} \cos(2\varphi)^2\right\}.\]
The expression in the last bracket can be seen to be positive for \(0<\varphi<\frac{\pi}{r+1}\) in the following way: Using the inequality
\[(r-1)\tan(r+1)\varphi > (r+1)\tan(r-1)\varphi,\]
gives 
\[\frac{d}{d\varphi} \frac{\sin(r-1)\varphi}{\sin(r+1)\varphi}>0.\]
This implies 
\[\frac{\sin(r-1)\varphi}{\sin(r+1)\varphi}>\frac{r-1}{r+1},\]
and we can conclude
\[(r+1)\cos(\varphi)^2 \sin(r-1)\varphi > (r-1)\sin(r+1)\varphi \cos(\varphi)^2> (r-1)\sin(r+1)\varphi \cos(2\varphi)^2.\]
Thus, all conditions of Theorem \ref{MSP} are satisfied (where we consider \(S(\varphi)\) as the saddle point instead of the origin) and by an application of (\ref{A}) we obtain
\[\int\limits_{U} e^{-n p(t_1,\ldots,t_r)} \tilde{Q}(t_1,\ldots,t_{r})dT=\left(\frac{2\pi}{n}\right)^{r/2} e^{-n p(S(\varphi))} \frac{\tilde{Q}(S(\varphi))}{\sqrt{\det\Hess p (S(\varphi))}} \left(1+o(1)\right),\]
as \(n\rightarrow \infty\). Using the definitions of \(p, \tilde{Q}\) and \(S(\varphi)\), we otain
\begin{align}\nonumber&\int\limits_{U} e^{-n p(t_1,\ldots,t_r)} \tilde{Q}(t_1,\ldots,t_{r})dT\\
&\label{sad1}=\left(\frac{2\pi}{n}\right)^{r/2}\frac{\left(1+w(\varphi)\right)^{\kappa}}{w(\varphi)^{\nu_1+\ldots+\nu_{r-1}}}
\left\{\left(a(\varphi) e^{i\varphi}\right)^r \left(r+1-(r-1)a(\varphi)^2 e^{2 i \varphi}\right) \right\}^{-1/2} \\
&\nonumber\times\left\{\exp\left\{a(\varphi)(r-1)e^{i\varphi}\right\}\left(1-b(\varphi)e^{i \theta}\right)^{-1}\left(1-\frac{\sin(r+1)\varphi}{b(\varphi)\sin(2\varphi)}e^{-i((r-1)\varphi+\theta)}\right)\right\}^n\left(1+o(1)\right),
\end{align}
as \(n\rightarrow \infty\). Now, taking into account that the contribution from the second saddle point \(-S(\varphi)\) will be the conjugate complex expression, computing the real part of (\ref{sad1}) and bearing in mind the prefactors coming from (\ref{int2}) will finally lead to the asymptotic form for the polynomials \(F_n (n^{r-1}x)\) as stated in the theorem.
\end{proof}

For the purpose of illustration we take a look at a plot for the case \(r=3\), \(\kappa=2\), \(\nu_1 =2\), \(\nu_2=5\) and \(n=150\). Showing the interval \(\left[\frac{2\pi}{13}, \frac{\pi}{6}\right]\), in Fig. 1 the normalized polynomial \(\tilde{F}_n\) (solid line) and the associated cosine approximant \(c_n(\varphi)\) (dashed line) are plotted, where we have
\begin{align*}&\tilde{F}_n(\varphi)\\ 
&=\frac{F_n \left(n^{r-1}\frac{\left(\sin{(r+1)\varphi}\right)^{(r+1)/2}}{\sin{2\varphi}\left(\sin{(r-1)\varphi}\right)^{(r-1)/2}}\right) \left((r+1)^2-2(r^2-1)a(\varphi)^2\cos{2\varphi}+(r-1)^2 a(\varphi)^4\right)^{1/4}}{\frac{2(-1)^n}{(2\pi)^{r/2}}\left(a(\varphi)n\right)^{-r/2-(\nu_1+\ldots+\nu_{r-1})}\left(1+2a(\varphi)\cos\varphi +a(\varphi)^2\right)^{\kappa/2} }\\
&~~\times \exp\left\{-n a(\varphi)(r-1)\cos\varphi\right\}\left(\frac{\sqrt{(1-a(\varphi)^2)^2+(2a(\varphi)\sin\varphi)^2}}{1+a(\varphi)^2-2a(\varphi)\cos\varphi}\right)^{-n},
\end{align*}
and
\[c_n(\varphi)=\cos{\left(n\,f(\varphi)+g(\varphi)\right)},\]
where \(f(\varphi)\) and \(g(\varphi)\) are defined in the statement of Theorem \ref{PRA}.

\begin{figure}[!htb]
\centering
\includegraphics[scale=0.62]{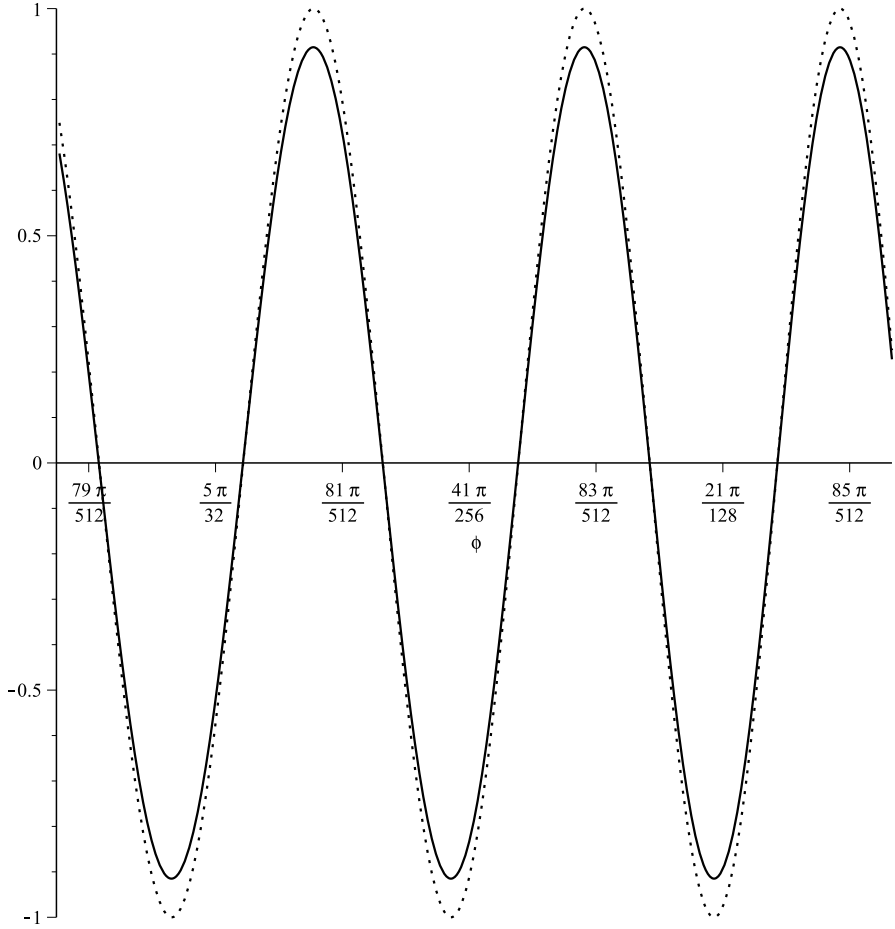}
\caption{The normalized polynomial \(\tilde{F}_n\) (solid line) and its cosine approximant \(c_n\) (dashed line).}
\label{Fig1}
\end{figure}

\bigskip

\section{Asymptotic zero distribution of the average characteristic polynomials}

The next aim is to study the behavior of the zeros of the rescaled polynomials \(F_n (n^{r-1}x)\). Therefore, like in Theorem \ref{PRA}, let the functions \(\sigma\) and \(f\) be defined by
\begin{equation}\label{sigma}\sigma:\left(0, \frac{\pi}{r+1}\right)\rightarrow \left(0,\frac{(r+1)^{(r+1)/2}}{2(r-1)^{(r-1)/2}}\right), \quad \sigma(\varphi)=\frac{\left(\sin{(r+1)\varphi}\right)^{(r+1)/2}}{\sin{2\varphi}\left(\sin{(r-1)\varphi}\right)^{(r-1)/2}},
\end{equation}
\begin{equation}\label{f}f:\left(0, \frac{\pi}{r+1}\right)\rightarrow \left(0, \pi\right), \quad f(\varphi)=\frac{\pi}{2}-(r-1)a(\varphi)\sin\varphi+\arctan\left(\frac{1-a(\varphi)^2}{2a(\varphi)\sin\varphi}\right).
\end{equation}
The function \(\sigma\) is a strictly decreasing bijection, whereas the function \(f\) is a strictly increasing bijection. Hence, the composition \(f\circ \sigma^{-1}\) is a strictly decreasing mapping from \(\left(0, \frac{(r+1)^{(r+1)/2}}{2(r-1)^{(r-1)/2}}\right)\) onto \(\left(0, \pi\right)\), which admits a continuous extension of the same kind to the interval \(\left[0, \frac{(r+1)^{(r+1)/2}}{2(r-1)^{(r-1)/2}}\right]\). Moreover, let the function \(V:\mathbb{R}\rightarrow [0,1]\) be defined by

\begin{equation}\label{V}V(x) = \begin{cases} 0 &\mbox{if } x \leq 0 \\
1-\frac{1}{\pi}\left(f\circ \sigma^{-1}\right)(x) & \mbox{if } 0<x<\frac{(r+1)^{(r+1)/2}}{2(r-1)^{(r-1)/2}}\\
1 &\mbox{if } x\geq \frac{(r+1)^{(r+1)/2}}{2(r-1)^{(r-1)/2}}.\end{cases}
\end{equation}
As it is not difficult to see that \(V\) is an increasing function on \(\mathbb{R}\) (with values in \([0, 1]\)) we can consider \(V\) as a probability distribution function (see Fig. 2 in the case \(r=3\)).

\begin{figure}[!htb]
\centering
\includegraphics[scale=0.62]{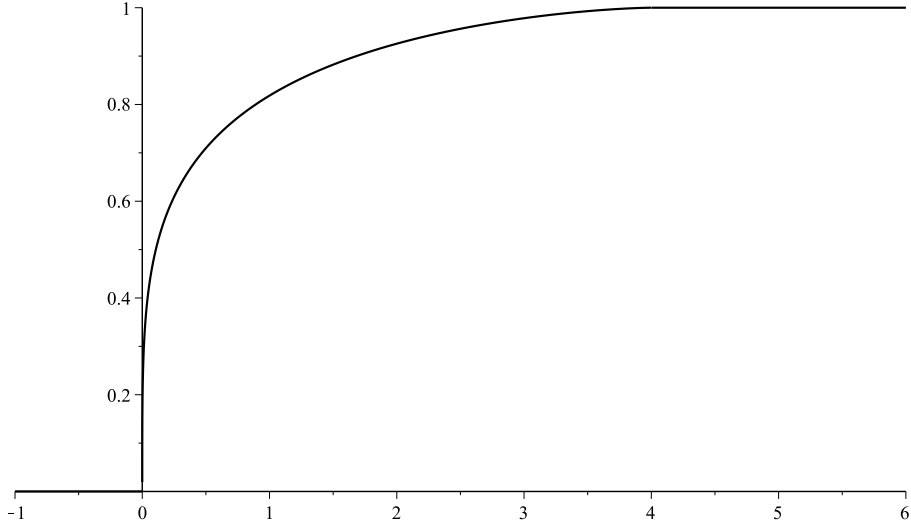}
\caption{The distribution function \(V\) in the case \(r=3\).}
\label{Fig2}
\end{figure}

In order to state the next theorem we introduce the class of Raney distributions, which is a natural generalization of the Fuss-Catalan distributions. For real \(\alpha\geq 1\), \(0<\beta\leq \alpha\) the Raney numbers \(R_{\alpha,\beta} (n)\) are defined by
\[R_{\alpha,\beta} (n)=\frac{\beta}{\alpha n +\beta} \binom{\alpha n+\beta}{n},\quad n=0,1,\ldots.\]
In \cite{Mlot} it is proved that these numbers form a moment sequence of some distribution with compact support on \([0,\infty)\) and the distributions are called Raney distributions \(R_{\alpha, \beta}\) (see, e.g., \cite{Penson}, \cite{Mlot2}, \cite{Forrester3}).

\begin{theorem}[Asymptotic zero distribution]\label{WA}Let \(r\in\mathbb{N}\), \(r>1\), \(\nu_1, \ldots, \nu_{r-1}, \kappa \in \mathbb{N}_0\) be arbitrary integers and let \((\mu_n)_n\) denote the sequence of normalized zero counting measures associated to the polynomials \(F_n (n^{r-1}x)\). Then the sequence \((\mu_n)_n\) converges in the weak-star sense to a unit measure \(\mu\) supported on \(\left[0, \frac{(r+1)^{(r+1)/2}}{2(r-1)^{(r-1)/2}}\right]\) which is defined by the distribution function \(V\) in (\ref{V}). Moreover, the limit measure \(\mu\) coincides with the Raney distribution \(R_{\frac{r+1}{2}, \frac{1}{2}}\).
\end{theorem}

\begin{proof} 
The main line of the proof follows similar arguments as the proof of Theorem 3.2 in \cite{Neuschel2}, therefore we only indicate some necessary modifications. Using the idea of counting the zeros of the cosine approximant in (\ref{PRA}) and involving the Portmanteau theorem and Helly's selection principle, it follows like in the first part of the proof of Theorem 3.2 in \cite{Neuschel2} that the measure \(\mu\) given by the distribution function \(V\) in (\ref{V}) indeed is the weak-star limit of the zero counting measures \(\mu_n\).

In order to show that this measure \(\mu\) coincides with the Raney distribution \(R_{\frac{r+1}{2}, \frac{1}{2}}\) we consider the Stieltjes transform of the latter
\[F(z)=\int\limits_{0}^{x^{\ast}}\frac{v(x)}{z-x}dx,\]
where \(v\) denotes its continuous density and we write \(x^{\ast}=\frac{(r+1)^{(r+1)/2}}{2(r-1)^{(r-1)/2}}\).

Next we use the fact (see, e.g., \cite{Forrester3}) that the function \(w(z)=zF(z)\) satisfies the equation
\begin{equation}\label{AE} w(z)^{r+1}-zw(z)^2+z=0.
\end{equation}
and admits an analytic continuation to \(\mathbb{C}\backslash [0,x^{\ast}]\), which has branch points exactly at the origin and at \(x^{\ast}\). Studying equation (\ref{AE}) shows that for \(z>x^{\ast}\) there are exactly two positive solutions, one of which converges to unity (this is \(w(z)\)) while the other solution tends to infinity like \(z^{\frac{1}{r-1}}\), as \(z\rightarrow \infty\) along the positive axis. These two solutions are connected to each other via the cut \((0,x^{\ast})\) and both branches converge to the value \(\left(\frac{r+1}{r-1}\right)^{\frac{1}{2}}\) as \(z\rightarrow x^{\ast}\), \(z>x^{\ast}\). Now, on the cut \((0,x^{\ast})\) we know from the proof of Theorem \ref{PRA} that using the parametrization
\[x=\sigma(\varphi)=\frac{\left(\sin{(r+1)\varphi}\right)^{(r+1)/2}}{\sin{2\varphi}\left(\sin{(r-1)\varphi}\right)^{(r-1)/2}},~~~0<\varphi<\frac{\pi}{r+1},\]
two roots of (\ref{AE}) are located at the points
\[w=a(\varphi) e^{i \varphi}\quad\quad\text{and}\quad\quad w=a(\varphi) e^{-i \varphi},\]
where \(a(\varphi)\) is defined in (\ref{a2}). Both solutions converge to \(\left(\frac{r+1}{r-1}\right)^{\frac{1}{2}}\) as \(\varphi \rightarrow 0\), from which we conclude that those solutions give the boundary values of \(w(z)\) as \(z\) approaches the cut from below or from above, respectively. Using the formula of Stieltjes-Perron we obtain for \(0<x<x^{\ast}\)
\begin{align*}v(x)&=\lim_{\epsilon \rightarrow 0+}\frac{1}{2\pi i} \left(F(x-i\epsilon)-F(x+i\epsilon)\right)\\
&=\lim_{\epsilon \rightarrow 0+}\frac{1}{2\pi i} \left(\frac{w(x-i\epsilon)}{x-i\epsilon}-\frac{w(x+i\epsilon)}{x+i\epsilon}\right).
\end{align*}
Putting \(x=\sigma(\varphi)\) we have for \(0<\varphi<\frac{\pi}{r+1}\)
\begin{align}\label{DEN}v(\sigma(\varphi))=\frac{\sin 2\varphi \sin \varphi (\sin (r-1)\varphi)^{\frac{r}{2}-1}}{\pi (\sin(r+1)\varphi)^{\frac{r}{2}}}.
\end{align}
On the other hand, deriving the distribution function \(V\) we obtain the following expression for its density for \(0<\varphi<\frac{\pi}{r+1}\)
\begin{equation}\label{DV}V'(\sigma(\varphi))=-\frac{1}{\pi}\frac{f'(\varphi)}{\sigma'(\varphi)}.
\end{equation}
Now, an elementary but cumbersome computation using (\ref{sigma}) and (\ref{f}) shows that the expressions (\ref{DEN}) and (\ref{DV}) coincide and thus the measure \(\mu\) is identified as the Raney distribution \(R_{\frac{r+1}{2}, \frac{1}{2}}\).
\end{proof}

The proof of Theorem \ref{WA} and especially the identity (\ref{DEN}) shows that after changing the coordinates from \(x\) to \(\varphi\) we obtain an elementary and explicit description for the density of the Raney distribution \(R_{\frac{r+1}{2}, \frac{1}{2}}\). We summarize this in the following theorem.

\begin{theorem}\label{C} Let \(v(x)\) denote the continuous density of the Raney distribution \(R_{\frac{r+1}{2}, \frac{1}{2}}\) defined on \(\left(0, \frac{(r+1)^{(r+1)/2}}{2(r-1)^{(r-1)/2}}\right)\). If
 \[x=\sigma(\varphi)=\frac{\left(\sin{(r+1)\varphi}\right)^{(r+1)/2}}{\sin{2\varphi}\left(\sin{(r-1)\varphi}\right)^{(r-1)/2}},\quad 0<\varphi<\frac{\pi}{r+1},\]
then we have
\[v(x)=\frac{\sin 2\varphi \sin \varphi (\sin (r-1)\varphi)^{\frac{r}{2}-1}}{\pi (\sin(r+1)\varphi)^{\frac{r}{2}}}.\]
In the same sense we also obtain an explicit and elementary expression for the corresponding distribution function
\[V(x)=\frac{1}{2}+\frac{(r-1)\sin\varphi}{\pi}\left(\frac{\sin(r+1)\varphi}{\sin(r-1)\varphi}\right)^{\frac{1}{2}}+\frac{1}{\pi}\arctan\left(\frac{\cos r\varphi}{\sqrt{\sin(r-1)\varphi \sin(r+1)\varphi}}\right).\]
\end{theorem}
\begin{remark} Recently in \cite{Forrester3}, Forrester and Liu used this method of parameterization to obtain similar forms of the densities for the general class of Raney distributions.
\end{remark}
\begin{remark} It is interesting to remark that the asymptotic zero distribution of the rescaled average characteristic polynomials \(F_n (n^{r-1}x)\) described in Theorem \ref{WA} and Theorem \ref{C} coincides with the macroscopic density of eigenvalues of the corresponding random matrices \(Z_r^{*}Z_r\) defined in (\ref{2}). This density can be constructed in free probability theory by means of the so-called free multiplicative convolution \(\boxtimes\) by
\[R_{\frac{r+1}{2}, \frac{1}{2}}=FC_{r-1} \boxtimes R_{1, \frac{1}{2}},\]
where \(FC_{r-1}\) denotes the Fuss-Catalan distribution of order \(r-1\) and the Raney distribution \(R_{1, \frac{1}{2}}\) coincides with the arcsine measure transformed to \([0,1]\) (see also \cite{Mlot}). In case that there are more than one truncated Haar distributed unitary matrix involved in the product (\ref{2}), the limiting distributions leave the general class of Raney distributions and enter the class of Jacobi polynomial moment measures (see \cite{Gawronski2}).

\end{remark}


\begin{thebibliography}{00}

\bibitem[1]{Adhi} K. Adhikari, N. Reddy, T. Reddy, K. Saha, Determinantal point processes in the plane from products of random matrices, preprint arXiv:1308.6817.
\bibitem[2]{Akemann1} G. Akemann, Z. Burda, Universal microscopic correlation functions for products of independent Ginibre matrices, J. Phys. A: Math. Theor. 45 (2012) 465201.
\bibitem[3]{Akemann2} G. Akemann, J. Ipsen, M. Kieburg, Products of rectangular random matrices: singular values and progressive scattering, Physical Review E 88(5): 052118.
\bibitem[4]{Burda1} Z. Burda, R. Janik, B. Waclaw, Spectrum of the product of independent random Gaussian matrices, Phys. Rev. E 81 (2010) 041132.
\bibitem[5]{Fedoryuk} M. Fedoryuk, Saddle-point method (Russian), Nauka, Moscow, 1977.
\bibitem[6]{Forrester0} P. Forrester, Eigenvalue statistics for product complex Wishart matrices, preprint arXiv:1401.2572.
\bibitem[7]{Forrester1} P. Forrester, Probability of all eigenvalues real for products of standard Gaussian matrices, J. Phys. A: Math. Theor. 47 065202.
\bibitem[8]{Forrester3} P. Forrester, D. Liu, Raney Distributions and Random Matrix Theory, preprint arXiv:1404.5759.
\bibitem[9]{Gawronski2} W. Gawronski, T. Neuschel, D. Stivigny, Jacobi polynomial moments and products of random matrices, preprint.
\bibitem[10]{Hardy} A. Hardy, Average Characteristic Polynomials of Determinantal Point Processes, Annales de l'Institut Henri Poincar\'{e}, to appear.
\bibitem[11]{KuijlaarsStivigny} A. Kuijlaars, D. Stivigny, Singular values of products of random matrices, Random Matrices: Theory and Applications (2014) DOI 10.1142/S2010326314500117.
\bibitem[12]{KuijlaarsZhang} A. Kuijlaars, L. Zhang, Singular values of products of Ginibre random matrices, multiple orthogonal polynomials and hard edge scaling limits, Commun. Math. Phys. (2014) DOI 10.1007/s00220-014-2064-3.
\bibitem[13]{Mlot} W. Mlotkowski, Fuss-Catalan numbers in noncommutative probability, Documenta Mathematica 15 (2010), 939--955.
\bibitem[14]{Mlot2} W. Mlotkowski, K. Penson, K. \.{Z}yczkowski, Densities of the Raney Distributions, Documenta Mathematica 18 (2013) 1573--1596.
\bibitem[15]{Neuschel} T. Neuschel, Apéry Polynomials and the multivariate Saddle Point Method, Constructive Approximation (2014) DOI 10.1007/s00365-014-9245-3.
\bibitem[16]{Neuschel2} T. Neuschel, Plancherel-Rotach formulae for average characteristic polynomials of products of Ginibre random matrices and the Fuss-Catalan distribution, Random Matrices: Theory and Applications
    Vol. 3 (1) (2014).
\bibitem[17]{Penson} K. Penson, K. \.{Z}yczkowski, Product of Ginibre matrices: Fuss-Catalan and Raney distributions, Phys. Rev. E 83 (6) (2011) 061118.
\bibitem[18]{Polya2} G. P\'{o}lya, G. Szeg\H{o}, Problems and Theorems in Analysis II, Springer, 1976.




\end{thebibliography}
\end{document}